\documentclass[12pt]{amsart}
\usepackage{amsfonts}
\usepackage{amsfonts,latexsym,rawfonts,amsmath,amssymb,amsthm}
\usepackage[plainpages=false]{hyperref}
\usepackage{cite}

\usepackage{graphicx}

\RequirePackage{color}

\numberwithin{equation}{section}

\newcommand{\beq}{\begin{equation}}
\newcommand{\eeq}{\end{equation}}
\newcommand{\beqs}{\begin{eqnarray*}}
\newcommand{\eeqs}{\end{eqnarray*}}
\newcommand{\beqn}{\begin{eqnarray}}
\newcommand{\eeqn}{\end{eqnarray}}
\newcommand{\beqa}{\begin{array}}
\newcommand{\eeqa}{\end{array}}

\newtheorem{definition}{Definition}
\newtheorem{lemma}{Lemma}
\newtheorem{remark}{Remark}

\newtheorem{theorem}{Theorem}
\newtheorem{example}{Example}

\title  {sensitivity and transitivity for the induced maps on symmetric product suspensions of a topological space}

\begin{document}

%\address{Hongbo Zeng: Department of Mathematical Sciences, Tsinghua University,  China}

%\email{zenghongbo@csust.edu.cn }

%\date{}

\bibliographystyle{plain}

%\tableofcontents
\maketitle

\baselineskip=15.8pt
\parskip=3pt

\centerline {\bf   \small Hongbo Zeng$^{a}$}
\centerline {$^{a}$School of Mathematics and Statistics, Changsha University of Science and Technology;}
\centerline {Hunan Provincial Key Laboratory of Mathematical Modeling and Analysis in Engineering,}
\centerline {Changsha 410114, Hunan, China}

\vskip20pt

\noindent {\bf Abstract}:
%\begin{abstract}
Given a nondegenerate compact perfect and Hausdorff topological space $X$,$n\in \mathbb{N}$ and a function $f:X\rightarrow X$, we consider the $n$-fold symmetric product of $X$, $F_n(X)$  and the induced function $F_n(f):F_n(X)\rightarrow F_n(X)$. If $n\geq2$, we consider the $n$-fold symmetric product suspension of $X$, $SF_n(X)$ and the induced function$SF_n(f):SF_n(X)\rightarrow SF_n(X)$.
In this paper, we study the relationships between the following statements: (1) $f\in \mathcal{M}$,(2) $F_n(f)\in \mathcal{M}$, and (3)$SF_n(f)\in \mathcal{M}$, where $\mathcal{M}$ is one of the following classes of map: sensitive, cofinitely sensitive, multi-sensitive, Z-transitive, quasi-periodic, accessible, indecomposable, multi-transitive, $\bigtriangleup$-transitive, $\bigtriangleup$-mixing, Martelli's chaos, Transitive, $F$-system,  $TT_{++}$, Touhey, two-sided transitive, fully exact, strongly transitive. These results improve and extend some existing ones.

%\end{abstract}
 \vskip20pt

\vskip20pt

\baselineskip=15.8pt
\parskip=3pt

%\newpage

%\tableofcontents
\maketitle

\baselineskip=15.8pt
\parskip=3.0pt

%\begin{abstract}

%\end{abstract}

%\keywords{iteration invariant; distributional chaos; Li-York chaos}

%\begin{multicols}{2}
\section{Introduction}
\noindent  Given a continuum (nonempty compact, connected metric space) $X$ and $n\in\mathbb{N}$, the $n$-fold symmetric product of $X$, $F_n(X)$, consists of all nonempty subsets of $X$ with at most $n$ points, and it is originally defined in \cite{a1}. If $n\geq2$, the $n$-fold symmetric product suspension of the continuum $X$, $SF_n(X)$, defined as the quotient space $F_n(X)/F_1(X)$, was introduced in \cite{a2}. Given a map $f:X\rightarrow X$, we consider the induced map $F_n(f):F_n(X)\rightarrow F_n(X)$ given by $F_n(f)(A)=f(A)$ for all $A\in F_n(X)$ \cite{a1}. If $n\geq2$, we consider the induced map $SF_n(f):SF_n(X)\rightarrow SF_n(X)$, this function is called $induced$ $map$ of $f$ on the $n$-fold symmetric product suspension of $X$ \cite{a3}. If $X$ is a continuum and $f:X\rightarrow X$ is a map, the dynamical system $(X,f)$ induces the dynamical systems $(F_n(X),F_n(f))$ and $(SF_n(X),SF_n(f))$.

\par
Since the introduction of these spaces, one of the most attractive problems and which is still highly
relevant for many researchers, consists of analyzing possible relationships between the conditions: (1) $f\in \mathcal{M}$,(2) $F_n(f)\in \mathcal{M}$, and (3)$SF_n(f)\in \mathcal{M}$, where $\mathcal{M}$ is a class of dynamic functions. In \cite{a4}, G. Higuera et al. study the relationships between the functions $f$ and $F_n(f)$, when one of them is: transitive, mixing, weakly mixing, Devaney chaotic or P property. In \cite{a5}, F. Barragn et al. study the relationships between the following statements: $f\in \mathcal{M}$, $F_n(f)\in \mathcal{M}$ and $SF_n(f)\in \mathcal{M}$, where $\mathcal{M}$ is one of the following classes of maps: exact, mixing, weakly mixing, transitive, totally transitive, strongly transitive, chaotic, minimal, irreducible, feebly open or turbulent. Continuing with the study of dynamic functions, in \cite{a6}, F. Barragn et al. study the relationships between the above statements, where $\mathcal{M}$ is one of the following classes of maps: mildly mixing, strongly product transitive, strongly exact transitive, fully exact, exact, exact transitive, strongly transitive, syndetically transitive or iteratively almost open.

\par
Recently, in \cite{a7} (2022), A. Illanes et al. solved most of the problems posed by F. Barragn et al. in \cite{a5}
and \cite{a6}, related to the properties of minimality, irreducibility, strong transitivity and turbulence.

\par
On the other hand, in recent years it has been of great interest to study the dynamic properties when the phase space is a topological space (\cite{a8,a9}). Regarding the dynamic functions and their induced functions, considering $X$ a topological space, in \cite{a10}, F. Barragn et al. study the relationship between the conditions: $f\in \mathcal{M}$ and $F_n(f)\in \mathcal{M}$, where $\mathcal{M}$ is one of the following classes of functions: exact, transitive, $Z$-transitive, $Z_+$-transitive, mixing, weakly mixing, chaotic, turbulent, strongly transitive, totally transitive,
orbit-transitive, strictly orbit-transitive, $\omega$-transitive, minimal, $TT_{++}$, semi-open or irreducible. In \cite{a11}, A. Rojas et al. work with topological spaces and analyzed the relationships between the following conditions:
$f\in \mathcal{M}$ and $F_n(f)\in \mathcal{M}$, where $\mathcal{M}$ is one of the following classes of functions: exact, transitive, strongly transitive, totally transitive, orbit-transitive, strictly orbit-transitive, $\omega$-transitive, mixing, weakly mixing, mild mixing, chaotic, exactly Devaney chaotic, minimal, backward minimal, totally minimal, $TT_{++}$, scattering, Touhey or an $F$-system.

\par
In this paper we begin the study of the $n$-fold symmetric product suspension, $SF_n(f)$, considering the phase space $X$ a nondegenerate, compact, perfect and Hausdorff topological space and $f:X\rightarrow X$
a function. We study the relationships between the following dynamical systems $(X,f)$, $(F_n(X),F_n(f))$ and $(SF_n(X),SF_n(f))$. Specifically, if $\mathcal{M}$ is one of the following classes of functions: almost transitive,
exact, mixing, transitive, totally transitive, strongly transitive, exactly Devaney chaotic, orbit-transitive,
an $F$-system, scattering, $TT_{++}$, Touhey, backward minimal, totally minimal, Property $P$, strong property
$P$ or two-sided transitive, we study the relationships between the following statements: (1) $f\in \mathcal{M}$,(2) $F_n(f)\in \mathcal{M}$, and (3)$SF_n(f)\in \mathcal{M}$.

\par
   This paper is organized as follows. In Section 2, we will first state some preliminaries, definitions and some lemmas. The main conclusions will be given in Section 3.

\section{Preparations and lemmas}
A (discrete) dynamical system is a pair $(X,f)$, where $X$ is a nondegenerate topological space and $f:X\rightarrow X$ is a function, $X$ is called the phase space. The symbols $\mathbb{Z}$, $\mathbb{Z_+}$ and $\mathbb{N}$ denotes the set of integers, the set of nonnegative integers and the set of positive integers, respectively. A compactum is a nondegenerate compact, perfect, Hausdorff topological space. A map is a continuous function.
Let $(X,f)$ be a dynamical system and let $x\in X$. The orbit of $x$ under $f$ is the set $orb(x,f)={f^k(x):
k\in\mathbb{Z_+}}$. A point $x$ of $X$ is a transitive point of the function $f$ if the set $orb(x,f)$ is dense in X. The set of transitive points of $f$ is denoted by $trans(f)$. The point $x$ is a fixed point of $f$ if $f(x)=x$. The point $x$ is a periodic point of $f$ if there exists $k\in\mathbb{N}$ such that $f^k(x)=x$. The set of periodic points of $f$ is denoted by $Per(f)$. The point $x$ is a quasi-periodic point of $f$ provided that for each open subset $U$ of $X$ such that $x\in U$, there exists $m\in\mathbb{N}$ such that $f^{km}(x)\in U$ for every $k\geq0$. The point $x$ is a recurrent point of $f$ if for each open subset $U$ of $X$ such that $x\in U$, there exists $m\in\mathbb{N}$ such that $f^{m}(x)\in U$. The point $x$ is a nonwandering point of $f$ provided that for all open subset $U$ in $X$ such that $x\in U$ there exists $m\in\mathbb{N}$ such that $f^m(U)\cap V\neq\emptyset$. A point $y$ in $X$ is an $\omega$-limit point of $x$ under $f$ if for every $k\in\mathbb{N}$ and for all open
subsets $U$ of $X$ such that $y\in U$, there exists a positive integer $m\geq k$ such that $f^m(x)\in U$. The set of
$\omega$-limit points of $x$ under $f$, is denoted by $\omega(x,f)$ and is called $\omega$-limit set of $x$. Given a subset $A$ of $X$, we say that $A$ is +invariant under $f$ if $f(A)\subseteq A$.

For the convenience of the following context, we denote $$N_{f}(U,V)=\{n\in \mathbb{N} \mid f^n(U)\cap V\neq\emptyset\},$$
denote $$n_{f}(U,V)=\{n\in \mathbb{N} : U\cap f^{-n}(V)\neq\emptyset\},$$
and denote $$N_{}(U,\delta)=\{n\in\mathbb{N}\mid \text{there exist} \ x,y\in U \ \text{such that } d(f_1^n(x),f_1^n(x))>\delta \}$$ for any nonempty open sets $U,V$ of $X$.

\begin{definition}

An autonomous system $(X, f)$  is said to be sensitive, if there is $\delta>0$ such that for any nonempty open set $U\subset X$, there exist $x,y\in U$ and $n\in \mathbb{N}$ such that $d(f^n(x),f^n(y))>\delta$.

An autonomous system $(X, f)$  is said to be cofinitely sensitive, if there is $\delta>0$ and $N\in \mathbb{N}$ such that for any nonempty open set $U\subset X$ and any $n>N$, there exist $x,y\in U$ such that $d(f^n(x),f^n(y))>\delta$.

An autonomous system $(X, f)$ is said to be multi-sensitive, if there exists $\delta>0$ such that for any $m\in \mathbb{N}$ and any nonempty open sets $U_1,U_2,...,U_m\subset X$, $\bigcap_{i=1}^mN_{}(U_i,\delta)\neq\emptyset$, where $\delta$ is called constant of sensitivity.
\end{definition}

\begin{definition}
Let $(X,f)$ be a dynamical system. $f$ is said to be $\mathbb{Z}$-transitive, if for any two non-empty subsets $U,V\subseteq X$, there exists $n\in \mathbb{Z}$ such that $f^n(U)\cap V\neq\emptyset$.

The autonomous system $(X, f)$  is said to be (topologically) transitive, if for any two non-empty subsets $U,V\subseteq X$, there exists $n\in \mathbb{N}$ such that $f^n(U)\cap V\neq\emptyset$.

The autonomous system $(X, f)$  is said to be weakly mixing, if for any four non-empty subsets $U_1,U_2,V_1,V_2\subseteq X$, there exists $n\in \mathbb{N}$ such that $f^n(U_1)\cap V_1\neq\emptyset$ and $f^n(U_2)\cap V_2\neq\emptyset$.

The autonomous system $(X, f)$  is said to be mixing, if for any two non-empty subsets $U,V\subseteq X$, there exists $N\in \mathbb{N}$ such that $f^n(U)\cap V\neq\emptyset$ for any $n\geq N$.

The autonomous system $(X, f)$  is said to be totally transitive, if for any $n\in \mathbb{N}$, $f^{n}$ is transitive.

The autonomous system $(X, f)$  is said to be strongly transitive, if for any non-empty subset $U\subseteq X$, there exists $M\in \mathbb{N}$ such that $\cup_{i=1}^Mf^i(U)=X$.

The autonomous discrete system $(X, f)$ is called multi-transitive if $f\times f^{2}\times\cdots\times f^{m}:X^m\rightarrow X^m$ is transitive for any $m\in \mathbb{N}$.

The autonomous discrete system $(X, f)$ is called $TT_++$, if for any pair of nonempty open subsets $U$ and $V$ of $X$, the set $n_f(U,V)$ is infinite.

The autonomous discrete system $(X, f)$ is called Touhey, if for every pair of nonempty open subsets $U$ and $V$ of $X$, there exist a periodic point $x\in U$ and $k\in \mathbb{Z_+}$ such that $f^k(x)\in V$.

The autonomous discrete system $(X, f)$ is called two-sided transitive, if $f$ is a homeomorphism and $\{f^n(x): n\in \mathbb{Z_+}\}$ is dense in $X$ for some $x\in X$.

The autonomous discrete system $(X, f)$ is called fully exact, if for every pair of nonempty open subsets $U$ and $V$ of $X$, there exist $k\in \mathbb{N}$ such that $int[f^(U)\cap f^k(V))]\neq \emptyset$.
\end{definition}

\begin{definition}
The autonomous system $(X, f)$  is said to be accessible, if for any $\varepsilon>0$ and any nonempty open set $U,V\subset X$, there exist $x\in U,y\in V$ and $n\in \mathbb{N}$ such that $d(f^n(x),f^n(y))<\varepsilon$.
\end{definition}

\begin{definition}\cite{a13}
 Let $(X,f)$ be a dynamical system. $f$ is said to be indecomposable if for any two $f$-invariant closed subsets $A,B$ of $X$ with $int(A)\neq\emptyset$ and $int(B)\neq\emptyset$, we have $int(A\cap B)\neq\emptyset$.
\end{definition}

\begin{definition}

The autonomous system $(X, f)$  is said to be $\bigtriangleup$-transitive, if for any  $n\in \mathbb{N}$, there exists a dense subset $Y$ of $X$ such that for each $y\in Y$, $(f^n(y),f^{2n}(y),...,f^{mn}(y):n\in\mathbb{Z_+})$ is dense in $=X^n$.

The autonomous system $(X, f)$  is said to be $\bigtriangleup$-mixing, if for any  $n\in \mathbb{N}$ and any infinite $B$ of $\mathbb{Z_+}$, there exists a dense subset $Y$ of $X$ such that for each $y\in Y$, $(f^n(y),f^{2n}(y),...,f^{mn}(y):n\in B)$ is dense in $=X^n$.

\end{definition}

\begin{definition}

Let $(X, f)$ be a compact metric space, the orbit of a point $x\in X$  is said to be unstable, if there exists $\delta>0$ such that for any neighborhood $U$ of $x$, there exists a point $y\in U$ and  $n\in \mathbb{Z_+}$ satisfying $d(f^n(x),f^n(y))>\arctan\delta$. The map $f$ is said to be  Martelli's chaos if there exists $x_0\in X$ such that the orbit of $x_0$ is dense in $X$ and unstable.

\end{definition}

\begin{definition}
 A point $x\in X$ is said to be transitive if the orbit of $x$ is dense in $X$. The autonomous system $(X, f)$  is said to be minimal if all the points of $X$ are transitive.
\end{definition}

\begin{definition}
An $F$-system, if $f$ is totally transitive and $Per(f)$ is dense in $X$.
\end{definition}

In \cite{a18}, the authors gave the definition of the metric on $SF_n(X)$. Let $X$ be a compactum and let $n\geq2$. Let $\mathbb{S}_n(X)=\{F_1(X)\cup \{A\}: A\in F_n(X)\}$. Note that  $\mathbb{S}_n(X)\subseteq\mathcal{C}_2(F_n(X))$. Let $G_n: SF_n(X)\twoheadrightarrow \mathbb{S}_n(X)$ be given by $G_n(\chi)=F_1(X)\cup q^{-1}(\chi)$. Then $G_n$ is a homeomorphism. Next, define the metric on $SF_n(X)$ $\rho_X^n: SF_n(f)\times SF_n(f)\rightarrow [0,\infty)$ by $\rho_X^n(\chi_1,\chi_2)=\mathcal{H}_X^2(G_n(\chi_1),(G_n(\chi_2))$, where $\mathcal{H}_X^2$ is the Hausdorff metric on $\mathcal{C}_2(F_n(X))$ induced by the Hausdorff metric $\mathcal{H}_X$ on $F_n(X)$.

Let $X$ be a compactum and let $n\geq2$. We consider the function $SF_n(f):SF_n(X)\rightarrow SF_n(X)$ given by
\[SF_n(f)(\chi)=\begin{cases}
q(F_n(f)(q^{-1}(\chi))), &   \text{if}\quad \chi\neq F_X, \\
F_X, &   \text{if}\quad \chi=F_X. \\
\end{cases}\]
Note that, $SF_n(f)$ is continuous, it is called induced map of $f$ on  $n$-fold symmetric product suspension\cite{a16}. Besides, it is easy to see that $q\circ F_n(f)=SF_n(f)\circ q$ and the following Remark

\begin{remark}\label{remark1}
$q\circ F_n(f)=SF_n(f)\circ q$. $q\circ F_n(f)^k=(SF_n(f))^k\circ q$. For any $A\in F_n(X)\setminus F_1(X)$, we have $q\circ F_n(f)^{-1}(A)=SF_n(f)^{-1}\circ q(A)$.
\end{remark}

\begin{remark}\label{remark2}
The space $SF_n(X)\setminus\{F_X\}$ is homeomorphic to $F_n(X)\setminus F_1(X)$.
\end{remark}

\begin{example}\cite[Example 4.6]{a5}

Let $f:S^1\rightarrow S^1$ be the map defined by $f(e^{2\pi i\theta})=e^{2\pi i(\theta+\alpha)}$, where $S^1$ is a unit circle on the complex plane, $\theta\in [0,1]$ and $\alpha$ is a irrational number. Let $\mathcal{M}$ is one of the following classes of map: exact, mixing, weakly mixing, transitive, totally transitive, strongly transitive, chaotic and minimal, then $SF_n(f)\notin \mathcal{M}$, but $f$ is transitive, totally transitive, strongly transitive, and minimal.

\end{example}

\begin{lemma}\label{yinli1}
For any $\Gamma$ be subset of $SF_n(X)$, we have that $F_n(f)\circ q^{-1}(\Gamma)\subseteq q^{-1}\circ SF_n(f)(\Gamma)$.
\end{lemma}
\begin{proof}
Since $q\circ F_n(f)\circ q^{-1}(\Gamma)=SF_n(f)\circ q\circ q^{-1}(\Gamma)=SF_n(f)(\Gamma)$, we have that $F_n(f)\circ q^{-1}(\Gamma)\subseteq q^{-1}\circ SF_n(f)(\Gamma)$.
\end{proof}

\begin{lemma}\cite[Theorem 3.13]{a15}\label{yinli2}
Let $(X,f)$ be a dynamical system. $f\times f$ is Z-transitive if and only if $f$ is weakly mixing.
\end{lemma}
\section{Main results}

The first three Theorems show that the relationships between the following statements: (1) $f$ is sensitive(cofinitely sensitive, multi-sensitive, respectively), (2)$F_n(f)$ is sensitive(cofinitely sensitive, multi-sensitive, respectively), (3) $SF_n(f)$ is sensitive(cofinitely sensitive, multi-sensitive, respectively).

\begin{theorem}
Let $X$ be a continuum, let $n$ be an integer greater than or equal to two, and let $f: X \rightarrow X$ be a map. Consider the following statements:

(1) $f$ is sensitive;

(2) $F_n(f)$ is sensitive;

(3) $SF_n(f)$ is sensitive.

Then (3) implies (2), and (2) implies (1).
\end{theorem}
\begin{proof}

Suppose that $SF_n(f)$ is sensitive with $\delta>0$ as a constant of sensitivity. We prove that $F_n(f)$ is also sensitive with $\delta>0$ as a constant of sensitivity. Let $\mathcal{U}$ be nonempty open subset of $F_n(X)$. By \cite[Lemma4.2]{a4}, there exist nonempty open subsets $U_1,U_2,...,U_n$ of $X$ such that $\langle U_1,U_2,...,U_n\rangle_n\subseteq \mathcal{U}$. For every $i\in\{1,2,...,n\}$, let $W_i$ be a nonempty open subset of $X$ such that $W_i\subseteq U_i$ and for any $i,j\in\{1,2,...,n\}$, $W_i\cap W_j\neq\emptyset$ if $i\neq j$. Note that $\langle U_1,U_2,...,U_n\rangle_n$ is nonempty open subsets of $F_n(X)$ such that $\langle W_1,W_2,...,W_n\rangle_n\subseteq \langle U_1,U_2,...,U_n\rangle_n\subseteq \mathcal{U}$ and $\langle W_1,W_2,...,W_n\rangle_n\cap F_1(X)=\emptyset$. By Remark \ref{remark2}, $q(\langle W_1,W_2,...,W_n\rangle_n)$ is nonempty open subsets of $SF_n(X)$ and $F_X\notin q(\langle W_1,W_2,...,W_n\rangle_n)$. Since $SF_n(f)$ be sensitive, there exist $\chi,\psi\in q(\mathcal{K}), m\in\mathbb{N}$, such that
$$\rho_X^n(SF_n(f)^m(\chi),SF_n(f)^m(\psi))>\delta.$$
Put $A=q^{-1}(\chi),B=q^{-1}(\psi)$. Then we have $\chi=q(A),\psi=q(B),A,B\in\mathcal{U}$. Therefore,

\begin{equation*}
\begin{aligned}
\delta&<\rho_X^n(SF_n(f)^m(\chi),SF_n(f)^m(\psi)) \\
&=d_H(F_1(X)\cup q^{-1}(SF_n(f)^m(\chi)),F_1(X)\cup q^{-1}(SF_n(f)^m(\psi))) \\
&=d_H(F_1(X)\cup q^{-1}(SF_n(f)^m(q(A))),F_1(X)\cup q^{-1}(SF_n(f)^m(q(B)))) \\
&=d_H(F_1(X)\cup F_n(f)^m(q^{-1}q(A)),F_1(X)\cup F_n(f)^m(q^{-1}q(B))) \\
&=d_H(F_1(X)\cup F_n(f)^m(A),F_1(X)\cup F_n(f)^m(B)) \\
&\leq d_H(F_n(f)^m(A),F_n(f)^m(B)).
\end{aligned}
\end{equation*}
This implies that $F_n(f)$ is sensitive.

Suppose that $F_n(f)$ is sensitive with sensitivity constant $\delta$. We see that $f$ is also sensitive with sensitivity constant $\delta$. For this end, let $\varepsilon>0$ be arbitrary. Let $x\in X$ and $U$ be the $\varepsilon$-neighborhood of $x$ in $X$. Then, as $\mathcal{U}=B_{d_H}(\{x\},\varepsilon)$ is an $\varepsilon$-neighborhood of $\{x\}$ in $F_n(X)$ and $F_n(f)$ is sensitive, there exist $A\in\mathcal{U}$ and $m\in\mathbb{N}$ such that $d_H(F_n(f)^m(A),F_n(f)^m(\{x\}))>\delta$. Hence, there exists $y\in A\subseteq U$ such that $d(f^m(y),f^m(x))>\delta$. This implies that $f$ is sensitive.

\end{proof}

\begin{theorem}
Let $X$ be a continuum, let $n$ be an integer greater than or equal to two, and let $f: X \rightarrow X$ be a map. Consider the following statements:

(1) $f$ is cofinitely sensitive;

(2)$F_n(f)$ is cofinitely sensitive;

(3)$SF_n(f)$ is cofinitely sensitive.

Then (1) and (2) are equivalent, and (3) implies (2).
\end{theorem}
\begin{proof}

Suppose that $F_n(f)$ is cofinitely sensitive with sensitivity constant $\delta$. We see that $f$ is also cofinitely sensitive with sensitivity constant $\delta$. For this end, let $\varepsilon>0$ be arbitrary. Let $x\in X$ and $U$ be the $\varepsilon$-neighborhood of $x$ in $X$. Then, as $\mathcal{U}=B_{d_H}(\{x\},\varepsilon)$ is an $\varepsilon$-neighborhood of $\{x\}$ in $F_n(X)$ and $F_n(f)$ is sensitive, there exist $A\in\mathcal{U}$ and $M\in\mathbb{N}$ such that $d_H(F_n(f)^m(A),F_n(f)^m(\{x\}))>\delta$ for any $m\geq M$. Hence, there exists $y\in A\subseteq U$ such that $d(f^m(y),f^m(x))>\delta$ for any $m\geq M$. This implies that $f$ is cofinitely sensitive.

Suppose that $f$ is cofinitely sensitive with sensitivity constant $\delta$. We see that $F_n(f)$ is cofinitely sensitive with sensitivity constant $\frac{\delta}{2}$. For this end, let $\varepsilon>0$ be arbitrary. Let $A=\{x_1,...,x_r\}\in F_n(X)$ and $B_{d_H}(A,\varepsilon)$ be the $\varepsilon$-neighborhood of $A$ in $F_n(X)$, where $r\leq n$. As $f$ is cofinite sensitive, for each $i=1,2,...,r$, there exist $y_i\in B(x_i,\varepsilon)$ and $m_i\in \mathbb{N}$ such that $d(f^{m_i}(x_i),f^{m_i}(y_i))>\delta$ for any $m\geq m_i$. Put $M=max\{m_i:1\leq i\leq r\}$. We will show that there exists $C\in B_{d_H}(A,\varepsilon)$ such that $d_H(F_n(f)^m(A),F_n(f)^m(C))>\frac{\delta}{2}$ for any $m\geq M$. Let $C=\{z_1,z_2,...,z_r\}$ where
\[z_i=\begin{cases}
y_i, &  d(f^m(x_1),f^m(x_i))\leq\frac{\delta}{2}, \\
x_i, &  otherwise,
\end{cases}\]
where $m\geq M$. Then it is easy to see that $d(f^m(x_1),f^m(z_i))>\frac{\delta}{2}$ for each $i$ and hence $d_H(F_n(f)^m(A),F_n(f)^m(C))>\frac{\delta}{2}$ for any $m\geq M$. Therefore, $F_n(f)$ is also cofinite sensitive.

Suppose that $SF_n(f)$ is cofinitely sensitive with $\delta>0$ as a constant of sensitivity. We prove that $F_n(f)$ is cofinitely sensitive with $\delta>0$ as a constant of sensitivity. Let $\mathcal{U}$ be nonempty open subset of $F_n(X)$. By \cite[Lemma4.2]{a4}, there exist nonempty open subsets $U_1,U_2,...,U_n$ of $X$ such that $\langle U_1,U_2,...,U_n\rangle_n\subseteq \mathcal{U}$. For every $i\in\{1,2,...,n\}$, let $W_i$ be a nonempty open subset of $X$ such that $W_i\subseteq U_i$ and for any $i,j\in\{1,2,...,n\}$, $W_i\cap W_j\neq\emptyset$ if $i\neq j$. Note that $\langle U_1,U_2,...,U_n\rangle_n$ is nonempty open subsets of $F_n(X)$ such that $\langle W_1,W_2,...,W_n\rangle_n\subseteq \langle U_1,U_2,...,U_n\rangle_n\subseteq \mathcal{U}$ and $\langle W_1,W_2,...,W_n\rangle_n\cap F_1(X)=\emptyset$. By Remark \ref{remark2}, $q(\langle W_1,W_2,...,W_n\rangle_n)$ is nonempty open subsets of $SF_n(X)$ and $F_X\notin q(\langle W_1,W_2,...,W_n\rangle_n)$. Since $SF_n(f)$ be  cofinitely sensitive, there exist $\chi,\psi\in q(\mathcal{K}), M\in\mathbb{N}$, such that for any $m\geq M$
$$\rho_X^n(SF_n(f)^m(\chi),SF_n(f)^m(\psi))>\delta.$$
Put $A=q^{-1}(\chi),B=q^{-1}(\psi)$. Then we have $\chi=q(A),\psi=q(B),A,B\in\mathcal{U}$. Therefore,

\begin{equation*}
\begin{aligned}
\delta&<\rho_X^n(SF_n(f)^m(\chi),SF_n(f)^m(\psi)) \\
&=d_H(F_1(X)\cup q^{-1}(SF_n(f)^m(\chi)),F_1(X)\cup q^{-1}(SF_n(f)^m(\psi))) \\
&=d_H(F_1(X)\cup q^{-1}(SF_n(f)^m(q(A))),F_1(X)\cup q^{-1}(SF_n(f)^m(q(B)))) \\
&=d_H(F_1(X)\cup F_n(f)^m(q^{-1}q(A)),F_1(X)\cup F_n(f)^m(q^{-1}q(B))) \\
&=d_H(F_1(X)\cup F_n(f)^m(A),F_1(X)\cup F_n(f)^m(B)) \\
&\leq d_H(F_n(f)^m(A),F_n(f)^m(B)),
\end{aligned}
\end{equation*}
which implies that $F_n(f)$ is cofinitely sensitive.

\end{proof}

\begin{theorem}
Let $X$ be a continuum, let $n$ be an integer greater than or equal to two, and let $f: X \rightarrow X$ be a map. Consider the following statements:

(1) $f$ is multi-sensitive;

(2)$F_n(f)$ is multi-sensitive;

(3)$SF_n(f)$ is multi-sensitive.

Then (1) and (2) are equivalent, and (3) implies (2).
\end{theorem}
\begin{proof}

Suppose that $F_n(f)$ is multi-sensitive with sensitivity constant $\delta$. We see that $f$ is also multi-sensitive with sensitivity constant $\delta$. For this end, let $\varepsilon>0$ and $m\in \mathbb{N}$ be arbitrary. Let $x_i\in X$ and $U_i$ be the $\varepsilon$-neighborhood of $x_i$ in $X$, where $i=1,2,...,m$. Then, as $\mathcal{U}_i=B_{d_H}(\{x_i\},\varepsilon)$ is an $\varepsilon$-neighborhood of $\{x_i\}$ in $F_n(X)$ and $F_n(f)$ is multi-sensitive, there exist $A_i\in\mathcal{U}_i$ and $k\in\mathbb{N}$ such that $d_H(F_n(f)^{k}(A),F_n(f)^{k}(\{x_i\}))>\delta$ for every $i=1,2,...,m$. Hence, there exists $y_i\in A_i\subseteq U_i$ such that $d(f^k(y_i),f^k(x_i))>\delta$ for every $i=1,2,...,m$. This implies that $f$ is multi-sensitive.

Suppose that $f$ is multi-sensitive with sensitivity constant $\delta$. We see that $F_n(f)$ is multi-sensitive with sensitivity constant $\frac{\delta}{2}$. For this end, let $\varepsilon>0$ and $m\in \mathbb{N}$  be arbitrary. Let $A_i=\{x_{i1},...,x_{ir_i}\}\in F_n(X)$ and $B_{d_H}(A_i,\varepsilon)$ be the $\varepsilon$-neighborhood of $A_i$ in $F_n(X)$, where $i=1,2,...,m$ and $r_i\leq n$ . As $f$ is multi-sensitive, for each $i=1,2,...,m$ and each $j=1,2,...,r_i$, there exist $y_{ij}\in B(x_{ij},\varepsilon)$ and $k\in \mathbb{N}$ such that $d(f^k(x_{ij}),f^k(y_{ij}))>\delta$. Next we will show that there exists $C_i\in B_{d_H}(A_i,\varepsilon)$ such that $d_H(F_n(f)^k(A_i),F_n(f)^k(C_i))>\frac{\delta}{2}$ for each $i=1,2,...,m$. Let $C_i=\{z_{i1},z_{i2},...,z_{ir_i}\}$ where
\[z_{ij}=\begin{cases}
y_{ij}, &  d(f^m(x_{i1}),f^m(x_{ij}))\leq\frac{\delta}{2}, \\
x_{ij}, &  otherwise,
\end{cases}\]
$i=1,2,...,m$ and $j=1,2,...,r_i$.
 Then it is easy to see that $d(f^k(x_{i1}),f^k(z_{ij}))>\frac{\delta}{2}$ for each $i=1,2,...,m$ and each $j=1,2,...,r_i$. Hence $d_H(F_n(f)^k(A_i),F_n(f)^k(C_i))>\frac{\delta}{2}$ for each $i=1,2,...,m$. Therefore, $F_n(f)$ is also multi-sensitive.

Suppose that $SF_n(f)$ is multi-sensitive with $\delta>0$ as a constant of sensitivity. We prove that $F_n(f)$ is multi-sensitive with $\delta>0$ as a constant of sensitivity. Let $m\in \mathbb{N}$  be arbitrary. Let $\mathcal{U}_i$ be nonempty open subset of $F_n(X)$, where $i=1,2,...,m$. By \cite[Lemma4.2]{a4}, there exist nonempty open subsets $U_{i1},U_{i2},...,U_{in}$ of $X$ such that $\langle U_{i1},U_{i2},...,U_{in}\rangle_n\subseteq \mathcal{U}_i$ for each $i\in\{1,2,...,m\}$. For every $i\in\{1,2,...,m\}$ and every $j\in\{1,2,...,n\}$, let $W_{ij}$ be a nonempty open subset of $X$ such that $W_{ij}\subseteq U_{ij}$ and for any $j,k\in\{1,2,...,n\}$, $W_{ij}\cap W_{ik}\neq\emptyset$ if $j\neq k$. Note that $\langle U_{i1},U_{i2},...,U_{in}\rangle_n$ is nonempty open subsets of $F_n(X)$ such that $\langle W_{i1},W_{i2},...,W_{in}\rangle_n\subseteq \langle U_{i1},U_{i2},...,U_{in}\rangle_n\subseteq \mathcal{U}_i$ and $\langle W_{i1},W_{i2},...,W_{in}\rangle_n\cap F_1(X)=\emptyset$ for each $i\in\{1,2,...,m\}$. By Remark \ref{remark2}, $q(\langle W_{i1},W_{i2},...,W_{in}\rangle_n)$ is nonempty open subsets of $SF_n(X)$ and $F_X\notin q(\langle W_{i1},W_{i2},...,W_{in}\rangle_n)$. Since $SF_n(f)$ be multi-sensitive, there exist $\chi_i,\psi_i\in q(\langle W_{i1},W_{i2},...,W_{in}\rangle_n)$ and $k\in\mathbb{N}$ such that
$$\rho_X^n(SF_n(f)^k(\chi_i),SF_n(f)^k(\psi_i))>\delta$$ for each $i\in\{1,2,...,m\}$.
Put $A_i=q^{-1}(\chi_i),B_i=q^{-1}(\psi_i)$. Then we have $\chi_i=q(A_i),\psi_i=q(B_i),A_i,B_i\in \mathcal{U}_i$. Therefore,

\begin{equation*}
\begin{aligned}
\delta&<\rho_X^n(SF_n(f)^k(\chi_i),SF_n(f)^k(\psi_i)) \\
&=d_H(F_1(X)\cup q^{-1}(SF_n(f)^k(\chi_i)),F_1(X)\cup q^{-1}(SF_n(f)^k(\psi_i))) \\
&=d_H(F_1(X)\cup q^{-1}(SF_n(f)^k(q(A_i))),F_1(X)\cup q^{-1}(SF_n(f)^k(q(B_i)))) \\
&=d_H(F_1(X)\cup F_n(f)^k(q^{-1}q(A_i)),F_1(X)\cup F_n(f)^k(q^{-1}q(B_i))) \\
&=d_H(F_1(X)\cup F_n(f)^k(A_i),F_1(X)\cup F_n(f)^k(B_i)) \\
&\leq d_H(F_n(f)^k(A_i),F_n(f)^k(B_i)).
\end{aligned}
\end{equation*}
This implies that $F_n(f)$ is multi-sensitive.

\end{proof}

In the following Theorem 4-Theorem 13, we study the relationships between the following statements: (1) $f\in \mathcal{M}$,(2) $F_n(f)\in \mathcal{M}$, and (3)$SF_n(f)\in \mathcal{M}$, where $\mathcal{M}$ is one of the following classes of map: Z-transitive, quasi-periodic, accessible, indecomposable, multi-transitive, $\bigtriangleup$-transitive, $\bigtriangleup$-mixing, Martelli's chaos.

\begin{theorem}
Let $X$ be a continuum, let $n$ be an integer greater than or equal to two, and let $f: X \rightarrow X$ be a map. Consider the following statements:

(1) $f$ is Z-transitive;

(2)$F_n(f)$ is Z-transitive;

(3)$SF_n(f)$ is Z-transitive.

Then (2) and (3) are equivalent, (2) implies (1), but (1) does not imply (2).
\end{theorem}
\begin{proof}
Suppose that $F_n(f)$ is Z-transitive, we prove that $SF_n(f)$ is Z-transitive. To this end, Let $\Gamma,\Lambda$  be nonempty open subsets of $SF_n(X)$. Since $q$ is continuous, $q^{-1}(\Gamma),q^{-1}(\Lambda)$ are nonempty open subsets of $F_n(X)$. By \cite[Lemma4.2]{a4}, there exist nonempty open subsets $U_1,U_2,...,U_n$ of $X$ such that $\langle U_1,U_2,...,U_n\rangle_n\subseteq q^{-1}(\Gamma)$. For every $i\in\{1,2,...,n\}$, let $W_i$ be a nonempty open subset of $X$ such that $W_i\subseteq U_i$ and for any $i,j\in\{1,2,...,n\}$, $W_i\cap W_j\neq\emptyset$ if $i\neq j$. Note that $\langle U_1,U_2,...,U_n\rangle_n$ is nonempty open subsets of $F_n(X)$ such that $\langle W_1,W_2,...,W_n\rangle_n\subseteq \langle U_1,U_2,...,U_n\rangle_n\subseteq \mathcal{U}$, $\langle W_1,W_2,...,W_n\rangle_n\cap F_1(X)=\emptyset$. Since $F_n(f)$ is Z-transitive, there exists $k\in\mathbb{Z}$ such that $F_n(f)^k(\langle W_1,W_2,...,W_n\rangle_n)\cap q^{-1}(\Lambda)\neq\emptyset$. It follows that there exists $A\in \langle W_1,W_2,...,W_n\rangle_n$ with $A\notin F_1(X)$ such that $F_n(f)^k(A)\in q^{-1}(\Lambda)$, that is, $f^k(A)\in q^{-1}(\Lambda)$. So we have that $q(A)\in \Gamma$ and $q\circ f^k(A)\in \Lambda$. By Remark \ref{remark1}, we have that $SF_n(f)^k\circ q(A)\in \Lambda$. Hence, $q(A)\in SF_n(f)^k(\Gamma)\cap \Lambda\neq\emptyset$, which implies that $SF_n(f)$ is Z-transitive.

Suppose that $SF_n(f)$ is Z-transitive, we prove that $F_n(f)$ is Z-transitive. For this end, let $\mathcal{U},\mathcal{V}$ be nonempty open subsets of $F_n(X)$. By \cite[Lemma4.2]{a4}, there exist nonempty open subsets $U_1,U_2,...,U_n$ and $V_1,V_2,...,V_n$ of $X$ such that $\langle U_1,U_2,...,U_n\rangle_n\subseteq \mathcal{U}$ and $\langle V_1,V_2,...,V_n\rangle_n\subseteq \mathcal{V}$. For every $i\in\{1,2,...,n\}$, let $W_i$ be a nonempty open subset of $X$ such that $W_i\subseteq U_i$ and for any $i,j\in\{1,2,...,n\}$, $W_i\cap W_j\neq\emptyset$ if $i\neq j$. Similarly, for each $i\in\{1,2,...,n\}$, let $O_i$ be a nonempty open subset of $X$ such that $O_i\subseteq V_i$ and for any $i,j\in\{1,2,...,n\}$, $O_i\cap O_j\neq\emptyset$ if $i\neq j$. Note that $\langle U_1,U_2,...,U_n\rangle_n$ and $\langle V_1,V_2,...,V_n\rangle_n$ are nonempty open subsets of $F_n(X)$ such that $\langle W_1,W_2,...,W_n\rangle_n\subseteq \langle U_1,U_2,...,U_n\rangle_n\subseteq \mathcal{U}$, $\langle O_1,O_2,...,O_n\rangle_n\subseteq \langle V_1,V_2,...,V_n\rangle_n\subseteq \mathcal{V}$, $\langle W_1,W_2,...,W_n\rangle_n\cap F_1(X)=\emptyset$, and $\langle O_1,O_2,...,O_n\rangle_n\cap F_1(X)=\emptyset$. By Remark \ref{remark2}, $q(\langle W_1,W_2,...,W_n\rangle_n)$ and $q(\langle O_1,O_2,...,O_n\rangle_n)$ are nonempty open subsets of $SF_n(X)$ such that $F_X\notin q(\langle W_1,W_2,...,W_n\rangle_n)$ and $F_X\notin q(\langle O_1,O_2,...,O_n\rangle_n)$. Since $SF_n(f)$ is Z-transitive, there exists $k\in\mathbb{Z}$ such that $SF_n(f)^k(q(\langle W_1,W_2,...,W_n\rangle_n))\cap q(\langle O_1,O_2,...,O_n\rangle_n)\neq\emptyset$. By Remark \ref{remark1}, it follows that $q\circ F_n(f)^k(\langle W_1,W_2,...,W_n\rangle_n)\cap q(\langle O_1,O_2,...,O_n\rangle_n)\neq\emptyset$. From Remark \ref{remark2}, we obtain that $F_n(f)^k(\langle W_1,W_2,...,W_n\rangle_n)\cap \langle O_1,O_2,...,O_n\rangle_n\neq\emptyset$, which implies that $F_n(f)^k(\mathcal{U})\cap \mathcal{V}\neq\emptyset$. Therefore, $F_n(f)$ is Z-transitive.

Suppose that $F_n(f)$ is Z-transitive, we prove that $f$ is Z-transitive. For this, let $U,V$ be nonempty open subsets of $X$. Then $\langle U\rangle_n,\langle V\rangle_n$ are nonempty open subsets of $F_n(X)$. Since $F_n(f)$ is Z-transitive, there exists $k\in\mathbb{Z}$ such that $F_n(f)^k(\langle U\rangle_n)\cap \langle V\rangle_n\neq\emptyset$. That is to say, there exists $A\in \langle U\rangle_n$ such that $F_n(f)^k(A)\in \langle V\rangle_n$. Since $A\subseteq U$, it follows that $f^k(U)\cap V\neq\emptyset$. Therefore, $f$ is Z-transitive.

By \cite[Theorem 4.10]{a5} and Theorem \ref{Thz}, we deduce that (1) does not imply (2).

\end{proof}

\begin{theorem}\label{Thz}
Let $X$ be a continuum, let $n$ be an integer greater than or equal to two, and let $f: X \rightarrow X$ be a map.
Then $F_n(f)$ is Z-transitive if and only if $f$ is weakly mixing.
\end{theorem}
\begin{proof}
By \cite[Theorem 4.11]{a5}, if $f$ is weakly mixing, then $F_n(f)$ is transitive, therefore, $F_n(f)$ is Z-transitive.

On the other hand, let $U_1,U_2,V_1,V_2$ be nonempty open subsets of $X$. It follows that $\langle U_1,U_2\rangle_n$ and $\langle V_1,V_2\rangle_n$ are nonempty open subsets of $F_n(X)$. Since $F_n(f)$ is Z-transitive, there exists $k\in\mathbb{Z}$ such that $F_n(f)^k(\langle U_1,U_2\rangle_n)\cap \langle V_1,V_2\rangle_n\neq\emptyset$. Then there exists $A\in\langle U_1,U_2\rangle_n$ such that $F_n(f)^k(A)\in \langle V_1,V_2\rangle_n$. It follows that $f^k(A)\cap V_1\neq\emptyset$ and $f^k(A)\cap V_2\neq\emptyset$. Take $x\in A\cap U_1$, $y\in A\cap U_2$, then we have $f^k(x)\cap V_1\neq\emptyset$ and $f^k(y)\cap V_2\neq\emptyset$, which implies that $f\times f$ is Z-transitive. By Lemma \ref{yinli2}, $f$ is weakly mixing.

\end{proof}

\begin{theorem}
Let $X$ be a compactum, let $n\geq2$, let $A\in F_n(X)$, and let $f: X \rightarrow X$ be a function. Consider the following statements:

(1)for each $x\in A$, $x$ is a quasi-periodic point of $f$.

(2)$A$ is a quasi-periodic point of $F_n(f)$.

(3)$q(A)$ is a quasi-periodic point of $SF_n(f)$.

Then the following hold: (1) implies (2), and (2) implies (3).
\end{theorem}
\begin{proof}
Suppose that for each $x\in A$, $x$ is a quasi-periodic point of $f$. Let $A=\{x_1,...,x_r\}$, and let $\varepsilon>0$. Since $x_i$ is a quasi-periodic point of $f$ for each $i=1,...,r$, there exists $m_i\in\mathbb{N}$ such that $f^{km_i}(x_i)\in B(x_i,\varepsilon)$ for any $k\in\mathbb{Z_+}$. Put $m=m_1m_2...m_r$. Then we have that $f^{km}(x_i)\in B(x_i,\varepsilon)$ for any $k\in\mathbb{Z_+}$. Therefore, $F_n(f)^{km}(A)\in B_{d_H}(A,\varepsilon)$ for any $k\in\mathbb{N}$. Thus, $A$ is a quasi-periodic point of $F_n(f)$.

Suppose that $A$ is a quasi-periodic point of $F_n(f)$. Let $\Omega$ be a nonempty open subset of $SF_n(X)$ such that $q(A)\in \Omega$. Then $q^{-1}(\Omega)$ be a nonempty open subset of $F_n(X)$ and $A\in q^{-1}(\Omega)$. Since $A$ is a quasi-periodic point of $F_n(f)$, there exists $m\in\mathbb{N}$ such that $f^{km}(A)\in q^{-1}(\Omega)$ for any $k\in\mathbb{Z_+}$. Further, $q\circ f^{km}(A)\in \Omega$ for any $k\in\mathbb{Z_+}$. By Remark \ref{remark1}, we have that $SF_n(f)^{km}\circ q(A)\in \Omega$. Therefore, $q(A)$ is a quasi-periodic point of $SF_n(f)$.

\end{proof}

\begin{theorem}
Let $X$ be a compactum, let $n\geq2$, let $A\in F_n(X)$, and let $f: X \rightarrow X$ be a function. Consider the following statements:

(1)$f$ is accessible.

(2)$F_n(f)$ is accessible.

(3)$SF_n(f)$ is accessible.

Then the following hold: (2) implies (1), and (2) implies (3).
\end{theorem}
\begin{proof}
Suppose that $F_n(f)$ is accessible, we prove that $f$ is accessible. Let $\varepsilon>0$ and let $U,V$ be nonempty open sets of $X$. Then $\langle U\rangle_n$ and $\langle V\rangle_n$ are nonempty open sets of $F_n(X)$. Since $F_n(f)$ is accessible, there exist $m\in\mathbb{N}$, $A\in \langle U\rangle_n,B\in \langle V\rangle_n$ such that $d_H(F_n(f)^n(A),F_n(f)^n(B))<\varepsilon$. Thus, there exist $x\in A, y\in B$ such that $d(f^m(x),f^m(y))\leq d_H(F_n(f)^m(x),F_n(f)^m(y))<\varepsilon$. Therefore, $f$ is accessible.

Suppose that $F_n(f)$ is accessible, we prove that $SF_n(f)$ is accessible. Let $\varepsilon>0$ and let $\Gamma,\Omega$ be nonempty open sets of $F_n(X)$. Then $q^{-1}(\Gamma)$ and $q^{-1}(\Omega)$ are nonempty open sets of $F_n(X)$. Since $F_n(f)$ is accessible, there exist $m\in\mathbb{N}$, $A\in q^{-1}(\Gamma),B\in q^{-1}(\Omega)$ such that $d_H(F_n(f)^m(A),F_n(f)^m(B))<\varepsilon$. Besides, we have $q(A)\in \Gamma$ and $q(B)\in \Omega$, and

\begin{equation*}
\begin{aligned}
&d_H(SF_n(f)^m(q(A)),SF_n(f)^m(q(B)))\\
=&d_H(q\circ F_n(f)^m(A),q\circ F_n(f)^m(B))\\
=&d_H(F_1(X)\cup q^{-1}\circ q\circ F_n(f)^m(A),F_1(X)\cup q^{-1}\circ q\circ F_n(f)^m(B))\\
\leq& d_H(F_n(f)^m(A),F_n(f)^m(B))\\
<&\varepsilon.
\end{aligned}
\end{equation*}

Therefore, $SF_n(f)$ is accessible.

\end{proof}

\begin{theorem}
Let $X$ be a continuum, let $n$ be an integer greater than or equal to two, and let $f: X \rightarrow X$ be a map. Consider the following statements:

(1) $f$ is indecomposable;

(2)$F_n(f)$ is indecomposable;

(3)$SF_n(f)$ is indecomposable.

Then (2) and (3) are equivalent, and (2) implies (1).
\end{theorem}
\begin{proof}
Suppose (2) holds, we prove (1) holds. Let $A,B$ are two $f$-invariant closed subsets of $X$  with $int(A)\neq\emptyset$ and $int(B)\neq\emptyset$. Then by \cite[Theorem3.2, Theorem 3.5]{a10}, $\langle A\rangle_n,\langle B\rangle_n$ are two $F_n(f)$-invariant closed subsets of $F_n(X)$  with $int(\langle A\rangle_n)\neq\emptyset$ and $int(\langle B\rangle_n)\neq\emptyset$. Since $F_n(f)$ is indecomposable,  we have $int(\langle A\rangle_n\cap \langle B\rangle_n)\neq\emptyset$, that is to say, there exists a nonempty open subset $\mathcal{U}$ of $F_n(X)$ such that $\mathcal{U}\subseteq \langle A\rangle_n\cap \langle B\rangle_n$. Put $W=\cup \mathcal{U}$, then $W$ is a nonempty open subset of $X$ by \cite[Theorem3.2]{a10}. Next, we will claim that $W\subseteq A\cap B$. For this, let $p\in W$, it follows that there exists $D\in \mathcal{U}$ such that $p\in D$. Note that $\mathcal{U}\subseteq \langle A\rangle_n\cap \langle B\rangle_n$. Therefore, $p\in A\cap B$. In consequence, $W\subseteq A\cap B$, which implies that $int(A\cap B)\neq\emptyset$. Hence, $f$ is indecomposable.

Suppose (2) holds, we show that (3) holds.  Let $\Gamma,\Lambda$ are two $SF_n(f)$-invariant closed subsets of $SF_n(X)$  with $int(\Gamma)\neq\emptyset$ and $int(\Lambda)\neq\emptyset$. Since $q$ is continuous, $q^{-1}(\Gamma),q^{-1}(\Lambda)$ are two closed subsets of $F_n(X)$  with $int(q^{-1}(\Gamma))\neq\emptyset$ and $int(q^{-1}(\Lambda))\neq\emptyset$. Besides, for any $y\in q^{-1}(\Lambda)$, there exist $z\in \Lambda$ such that
$y=q^{-1}(z)$. Notice that $$F_n(f)(y)=F_n(f)\circ q^{-1}(z)\in F_n(f)\circ q^{-1}(\Lambda)\subseteq q^{-1}\circ SF_n(f)(\Lambda)\subseteq q^{-1}(\Lambda),$$
which implies that $q^{-1}(\Gamma)$ is $F_n(f)$-invariant. As the same way, $q^{-1}(\Gamma)$ is $F_n(f)$-invariant.
Since $F_n(f)$ is indecomposable,  we have $int(q^{-1}(\Gamma)\cap q^{-1}(\Lambda))\neq\emptyset$, that is to say, there exists a nonempty open subset $\mathcal{U}$ of $F_n(X)$ such that $\mathcal{U}\subseteq q^{-1}(\Gamma)\cap q^{-1}(\Lambda)$. Then $q(\mathcal{U})\subseteq \Gamma\cap \Lambda$ and there exist nonempty open subsets $U_1,U_2,...,U_n$ of $X$ such that $\langle U_1,U_2,...,U_n\rangle_n\subseteq \mathcal{U}$ and $U_i\cap U_j=\emptyset$ if $i\neq j$ for each $i,j=1,2,...,n$. Hence by remark, $q\langle U_1,U_2,...,U_n\rangle_n$ nonempty open subset of $SF_n(X)$, which implies that $int(q(\mathcal{U}))\neq\emptyset$. This and the fact that $q(\mathcal{U})\subseteq \Gamma\cap \Lambda$ reduce that $int(\Gamma\cap \Lambda)\neq\emptyset$. Hence, $SF_n(f)$ is indecomposable.

Suppose (3) holds, we see that (2) holds. Let $\mathcal{U},\mathcal{V}$ are two $F_n(f)$-invariant closed subsets of $F_n(X)$  with $int(\mathcal{U})\neq\emptyset$ and $int(\mathcal{V})\neq\emptyset$. Then $q(\mathcal{U}),q(\mathcal{V})$ are two closed subsets of $SF_n(X)$  and $int(q(\mathcal{U}))\neq\emptyset$ and $int(q(\mathcal{V}))\neq\emptyset$. Besides, since $SF_n(f)\circ q(\mathcal{U})=q\circ F_n(f)\mathcal{U}\subseteq q(\mathcal{U})$, $q(\mathcal{U})$ is $SF_n(f)$-invariant. In the same way, $q(\mathcal{V})$ is $SF_n(f)$-invariant.
Suppose that $SF_n(f)$ is indecomposable. Then $int(q(\mathcal{U})\cap q(\mathcal{V}))\neq\emptyset$. This means that there exists a nonempty open subset $\Gamma$ of $SF_n(X)$ such that $\Gamma\subseteq q(\mathcal{U})\cap q(\mathcal{V})$. Since $q$ is continuous, $q^{-1}(\Gamma)$ is a nonempty open subset of $F_n(X)$ and $q^{-1}(\Gamma)subsete \mathcal{U}\cap \mathcal{V}$, it follows that $int(\mathcal{U}\cap \mathcal{V})\neq\emptyset$. Hence, $F_n(f)$ is indecomposable.

\end{proof}

\begin{theorem}
Let $X$ be a compactum, let $n\geq2$, let $A\in F_n(X)$, and let $f: X \rightarrow X$ be a function. Consider the following statements:

(1)$f$ is multi-transitive.

(2)$F_n(f)$ is multi-transitive.

(3)$SF_n(f)$ is multi-transitive.

Then (2) and (3) are equivalent, and (2) implies (1).
\end{theorem}
\begin{proof}
Suppose that $F_n(f)$ is multi-transitive, we prove that $SF_n(f)$ is multi-transitive. Let $m\in\mathbb{N}$, we see that $SF_n(f)\times SF_n(f)^2\times\cdot\cdot\cdot\times SF_n(f)^m: SF_n(X)^m\rightarrow SF_n(X)^m$ is transitive.
Let $\Gamma_1,\Gamma_2,...,\Gamma_m,\Lambda_1,\Lambda_2,...,\Lambda_m$  be any collection of nonempty open subsets of $SF_n(X)$. Since $q$ is continuous, $q^{-1}(\Gamma_1),q^{-1}(\Gamma_2),...,q^{-1}(\Gamma_m),q^{-1}(\Lambda_1),q^{-1}(\Lambda_2),...,q^{-1}(\Lambda_m)$ are collection of nonempty open subsets of $F_n(X)$. Since $F_n(f)$ is multi-transitive, there exists $k\in\mathbb{N}$ such that $F_n(f)^{ki}(q^{-1}(\Gamma_i))\cap q^{-1}(\Lambda_i)\neq\emptyset$ for each $i\in\{1,2,..,m\}$. It follows that there exists $A_i\in q^{-1}(\Gamma_i)$ such that $F_n(f)^{ki}(A_i)\in q^{-1}(\Lambda_i)$, that is, $f^{ki}(A_i)\in q^{-1}(\Lambda_i)$ for each $i\in\{1,2,..,m\}$. So we have that $q(A_i)\in \Gamma_i$ and $q\circ f^{ki}(A_i)\in \Lambda_i$ for each $i\in\{1,2,..,m\}$. By Remark \ref{remark1}, we have that $SF_n(f)^{ki}\circ q(A_i)\in \Lambda_i$ for each $i\in\{1,2,..,m\}$. Hence, $q(A_i)\in SF_n(f)^{ki}(\Gamma_i)\cap \Lambda_i\neq\emptyset$ for each $i\in\{1,2,..,m\}$, which implies that $SF_n(f)\times SF_n(f)^2\times\cdot\cdot\cdot\times SF_n(f)^m: SF_n(X)^m\rightarrow SF_n(X)^m$ is transitive. Therefore, $SF_n(f)$ is multi-transitive.

Suppose that $SF_n(f)$ is multi-transitive, we prove that $F_n(f)$ is multi-transitive. For this end, let $m\in\mathbb{N}$, let $i\in\{1,2,..,m\}$ and let $\mathcal{U}_i,\mathcal{V}_i$ be any collection of nonempty open subsets of $F_n(X)$. By \cite[Lemma4.2]{a4}, there exist nonempty open subsets $U_{i1},U_{i2},...,U_{in}$ and $V_{i1},V_{i2},...,V_{in}$ of $X$ such that $\langle U_{i1},U_{i2},...,U_{in}\rangle_n\subseteq \mathcal{U}_i$ and $\langle V_{i1},V_{i2},...,V_{in}\rangle_n\subseteq \mathcal{V}_i$. For every $j\in\{1,2,...,n\}$, let $W_{ij}$ be a nonempty open subset of $X$ such that $W_{ij}\subseteq U_{ij}$ and for any $j,k\in\{1,2,...,n\}$, $W_{ij}\cap W_{ik}\neq\emptyset$ if $j\neq k$. Similarly, for each $j\in\{1,2,...,n\}$, let $O_{ij}$ be a nonempty open subset of $X$ such that $O_{ij}\subseteq V_{ij}$ and for any $j,k\in\{1,2,...,n\}$, $O_{ij}\cap O_{ik}\neq\emptyset$ if $j\neq k$. Note that $\langle U_{i1},U_{i2},...,U_{in}\rangle_n$ and $\langle V_{i1},V_{i2},...,V_{in}\rangle_n$ are nonempty open subsets of $F_n(X)$ such that $\langle W_{i1},W_{i2},...,W_{in}\rangle_n\subseteq \langle U_{i1},U_{i2},...,U_{in}\rangle_n\subseteq \mathcal{U}_i$, $\langle O_{i1},O_{i2},...,O_{in}\rangle_n\subseteq \langle V_{i1},V_{i2},...,V_{in}\rangle_n\subseteq \mathcal{V}_i$, $\langle W_{i1},W_{i2},...,W_{in}\rangle_n\cap F_1(X)=\emptyset$, and $\langle O_{i1},O_{i2},...,O_{in}\rangle_n\cap F_1(X)=\emptyset$. By Remark \ref{remark2}, $q(\langle W_{i1},W_{i2},...,W_{in}\rangle_n)$ and $q(\langle O_{i1},O_{i2},...,O_{in}\rangle_n)$ are nonempty open subsets of $SF_n(X)$ such that $F_X\notin q(\langle W_{i1},W_{i2},...,W_{in}\rangle_n)$ and $F_X\notin q(\langle O_{i1},O_{i2},...,O_{in}\rangle_n)$. Since $SF_n(f)$ is multi-transitive, there exists $l\in\mathbb{N}$ such that $SF_n(f)^{li}(q(\langle W_1,W_2,...,W_n\rangle_n))\cap q(\langle O_1,O_2,...,O_n\rangle_n)\neq\emptyset$. By Remark \ref{remark1}, it follows that $q\circ F_n(f)^l(\langle W_{i1},W_{i2},...,W_{in}\rangle_n)\cap q(\langle O_{i1},O_{i2},...,O_{in}\rangle_n)\neq\emptyset$. From Remark \ref{remark2}, we obtain that $F_n(f)^{li}(\langle W_{i1},W_{i2},...,W_{in}\rangle_n)\cap \langle O_{i1},O_{i2},...,O_{in}\rangle_n\neq\emptyset$, which implies that $F_n(f)^{li}(\mathcal{U}_i)\cap \mathcal{V}_i\neq\emptyset$. Therefore, $F_n(f)$ is multi-transitive.

Suppose that $F_n(f)$ is multi-transitive, we prove that $f$ is multi-transitive. For this, let $m\in\mathbb{N}$, let $i\in\{1,2,..,m\}$ and let $U_i,V_i$ be nonempty open subsets of $X$. Then $\langle U_i\rangle_n,\langle V_i\rangle_n$ are nonempty open subsets of $F_n(X)$. Since $F_n(f)$ is multi-transitive, there exists $k\in\mathbb{N}$ such that $F_n(f)^{ki}(\langle U_i\rangle_n)\cap \langle V_i\rangle_n\neq\emptyset$. That is to say, there exists $A_i\in \langle U_i\rangle_n$ such that $F_n(f)^{ki}(A_i)\in \langle V_i\rangle_n$. Since $A_i\subseteq U_i$, it follows that $f^{ki}(U_i)\cap V_i\neq\emptyset$. Therefore, $f$ is multi-transitive.

\end{proof}

\begin{theorem}\label{dldeta}
Let $X$ be a continuum, let $n$ be an integer greater than or equal to two, and let $f: X \rightarrow X$ be a map. Consider the following statements:

(1) $f$ is $\bigtriangleup$-transitive;

(2)$F_n(f)$ is $\bigtriangleup$-transitive;

(3)$SF_n(f)$ is $\bigtriangleup$-transitive.

Then (2) implies (1), and (2) implies (3).
\end{theorem}
\begin{proof}
We first prove that (2) implies (1). For any $m\in\mathbb{Z_+}$, since $F_n(f)$ is $\bigtriangleup$-transitive,
there exist a subset $\mathcal{A}\subseteq F_n(X)$ with $\overline{\mathcal{A}}=F_n(X)$ such that
$$\overline{(F_n(f)^n(A),F_n(f)^{2n}(A),...,F_n(f)^{mn}(A):n\in\mathbb{Z_+})}=F_n(X)^m$$ for any $A\in \mathcal{A}$. By \cite[Theorem 3.14]{a10}, we have $\overline{\cup\mathcal{A}}=X$. Denote $Y=\cup\mathcal{A}$. Then for any $y\in Y$, there exists $A\in \mathcal{A}$ such that $y\in A$. For any nonempty open subset $V$ of $X^m$, there exist nonempty open subsets $V_1,V_2,...,V_m\subseteq X$ such that $(V_1,V_2,...,V_m)\subseteq V$. Hence  $(\langle V_1\rangle_n,\langle V_2\rangle_n,...,\langle V_m\rangle_n)$ is a nonempty open subset of $F_n(X)^m$. Hence, there exist $k\in\mathbb{Z_+}$ such that $F_n(f)^{ki}(A)\in \langle V_i\rangle_n$ for each $i=1,2,...,m$. That is to say, $f^{ki}(A)\in V_i$ for each $i=1,2,...,m$. Therefore, $f^{ki}(y)\in V_i$ for each $i=1,2,...,m$, which implies that  $$\overline{(f^n(y),f^{2n}(y),...,f^{mn}(y):n\in\mathbb{Z_+})}=X^m.$$
Thus $f$ is $\bigtriangleup$-transitive.

Next, we prove that (2) implies (3). For any $m\in\mathbb{Z_+}$, since $F_n(f)$ is $\bigtriangleup$-transitive,
there exist a subset $\mathcal{A}\subseteq F_n(X)$ with $\overline{\mathcal{A}}=F_n(X)$ such that
$$\overline{(F_n(f)^n(A),F_n(f)^{2n}(A),...,F_n(f)^{mn}(A):n\in\mathbb{Z_+})}=F_n(X)^m$$ for any $A\in \mathcal{A}$.
Since $q$ is  surjective, $\overline{q(\mathcal{A})}=SF_n(X)$. Denote $W=q(\mathcal{A})$. For any $\chi\in W$, there exists $A\in \mathcal{A}$ such that $\chi=q(A)$. For any nonempty open subset $\mathcal{V}$ of $SF_n(X)^m$, there exist nonempty open subsets $\mathcal{V}_1,\mathcal{V}_2,...,\mathcal{V}_m\subseteq SF_n(X)$ such that $(\mathcal{V}_1,\mathcal{V}_2,...,\mathcal{V}_m)\subseteq \mathcal{V}$. Hence  $(q^{-1}\mathcal{V}_1,q^{-1}\mathcal{V}_2,...,q^{-1}\mathcal{V}_m)$ is a nonempty open subset of $F_n(X)^m$. Hence, there exist $k\in\mathbb{Z_+}$ such that $F_n(f)^{ki}(A)\in q^{-1}\mathcal{V}_i$ for each $i=1,2,...,m$. That is to say, $q\circ F_n(f)^{ki}(A)\in \mathcal{V}_i$ for each $i=1,2,...,m$. Therefore, $SF_n(f)^{ki}(\chi)=F_n(f)^{ki}(q(A))\in \mathcal{V}_i$ for each $i=1,2,...,m$, which implies that  $$\overline{(SF_n(f)^n(\chi),SF_n(f)^{2n}(\chi),...,SF_n(f)^{mn}(\chi):n\in\mathbb{Z_+})}=SF_n(X)^m.$$
Thus $SF_n(f)$ is $\bigtriangleup$-transitive.

\end{proof}

\begin{theorem}
Let $X$ be a continuum, let $n$ be an integer greater than or equal to two, and let $f: X \rightarrow X$ be a map. Consider the following statements:

(1) $f$ is $\bigtriangleup$-mixing;

(2)$F_n(f)$ is $\bigtriangleup$-mixing;

(3)$SF_n(f)$ is $\bigtriangleup$-mixing.

Then (2) implies (1), and (2) implies (3).
\end{theorem}
\begin{proof}
The proof is similar to the proof in Theorem \ref{dldeta}.

\end{proof}

\begin{theorem}\label{thdj}
Let $X$ be a compactum, let $n\geq2$, and let $f: X \rightarrow X$ be a function. Then the following are equivalent:

(1)$f$ is weakly mixing.

(2)$F_n(f)$ is weakly mixing.

(3)$SF_n(f)$ is weakly mixing.

(4)$F_n(f)$ is totally transitive.

(5)$SF_n(f)$ is totally transitive.

(6)$F_n(f)$ is transitive.

(7)$SF_n(f)$ is transitive.

\end{theorem}
\begin{proof}
Since $X$ is compact, $F_n(X)$ is compact. Then by \cite[Theorem 3.3]{a17}, weakly mixing of $F_n(f)$ implies total transitivity of $F_n(f)$. Besides, Note that total transitivity of $F_n(f)$ implies transitivity of $F_n(f)$ and that transitivity of $F_n(f)$ is equivalent to weakly mixing of $F_n(f)$. Therefore, (3), (4) and (6) are equivalent.  With the similar argument, we can get the result.

\end{proof}

\begin{theorem}
Let $X$ be a compactum, and let $f: X \rightarrow X$ be a function. If $F_n(f)$ is Martelli's chaos, then $f$ is Martelli's chaos.
\end{theorem}
\begin{proof}
Since $F_n(f)$ is Martelli's chaos, there exists a transitive point $A\in F_n(X)$ and the orbit of $A$ is unstable in $F_n(X)$. It follows from Theorem \ref{dltr} that $x$ is a transitive point in $X$ for any $x\in A$. For a other hand, since $A$ is unstable in $F_n(X)$, there exists $\delta>0$ such that for any $\varepsilon>0$, there exist a point $B\in B_{d_H}(A,\varepsilon)$ and $n\in \mathbb{Z_+}$ satisfying $d_H(f^n(A),f^n(B))>\delta$. We claim that there exist $x_0\in A$ and $y_0\in B$ such that $y_0\in B(x_0,\varepsilon)$ and $d(f^n(x_0),f^n(y_0))>\delta$. If not, then for any $x\in A$ and any $y\in B$ with $y\in B(x,\varepsilon)$, $d(f^n(x),f^n(y))\leq\delta$. It follows that $d_H(f^n(A),f^n(B))\leq\delta$, which is a contradiction. Therefore, $x_0$ is a unstable point. Thus, $f$ is Martelli's chaos.

\end{proof}

It is well-known that weakly mixing dynamical systems have sensitive dependence on initial conditions. Since sensitive implies unstable. Therefore, weakly mixing dynamical systems have unstable orbit. And it is well-known that when $X$ is a compact metric space with no isolated points, transitive and point transitive are equivalent. Therefore, we can establish the following:

\begin{theorem}
Let $X$ be a compactum with no isolated point, let $n\geq2$, and let $f: X \rightarrow X$ be a function. Then the following are equivalent:

(1)$f$ is weakly mixing.

(2)$F_n(f)$ is weakly mixing.

(3)$SF_n(f)$ is weakly mixing.

(4)$F_n(f)$ is totally transitive.

(5)$SF_n(f)$ is totally transitive.

(6)$SF_n(f)$ is transitive.

(7)$SF_n(f)$ is transitive.

(8)$F_n(f)$ is Martelli's chaos.

\end{theorem}

\begin{proof}
By \cite[Theorem 3.16]{a10}, since $X$ is compact with no isolated points, then $F_n(X)$ is also compact with no isolated points. Then (1) and (8) are equivalent. The rest of equivalence from Theorem \ref{thdj}.

\end{proof}

The next eight Theorems extend Theorem 4.7, Theorem 4.8,Theorem 4.11, Theorem 4.12, Theorem 4.14, Theorem 4.17, Theorem 4.18 and Theorem 4.21 in \cite{a16}.
\begin{theorem}\label{dltr}
Let $X$ be a compactum, let $n\geq2$, let $B\in F_n(X)$, and let $f: X \rightarrow X$ be a function. Consider the following statements:

(1)for each $x\in B$, $x$ is a transitive point of $f$.

(2)$B$ is a transitive point of $F_n(f)$.

(3)$q(B)$ is a transitive point of $SF_n(f)$.

Then the following hold: (2) and (3) are equivalent, (2) implies (1), but (1) does not imply (2).
\end{theorem}
\begin{proof}
By \cite[Theorem 4.7]{a16}, we need only to prove that (3) implies (2). Let $\mathcal{U}$ be nonempty open subsets of $F_n(X)$. By \cite[Lemma4.2]{a4}, there exist nonempty open subsets $U_1,U_2,...,U_n$ of $X$ such that $\langle U_1,U_2,...,U_n\rangle_n\subseteq \mathcal{U}$. For every $i\in\{1,2,...,n\}$, let $W_i$ be a nonempty open subset of $X$ such that $W_i\subseteq U_i$ and for any $i,j\in\{1,2,...,n\}$, $W_i\cap W_j\neq\emptyset$ if $i\neq j$. Note that $\langle U_1,U_2,...,U_n\rangle_n$ is nonempty open subsets of $F_n(X)$ such that $\langle W_1,W_2,...,W_n\rangle_n\subseteq \langle U_1,U_2,...,U_n\rangle_n\subseteq \mathcal{U}$ and $\langle W_1,W_2,...,W_n\rangle_n\cap F_1(X)=\emptyset$. By Remark \ref{remark2}, $q(\langle W_1,W_2,...,W_n\rangle_n)$ is nonempty open subsets of $SF_n(X)$ and $F_X\notin q(\langle W_1,W_2,...,W_n\rangle_n)$. By hypothesis, $orb(q(B),SF_n(f))\cap q(\langle W_1,W_2,...,W_n\rangle_n)\neq\emptyset$. Thus, there exist $k\in\mathbb{Z_+}$ and $\chi\in q(\langle W_1,W_2,...,W_n\rangle_n)$ such that $SF_n(f)^k(q(B))\in q(\langle W_1,W_2,...,W_n\rangle_n)$ and $SF_n(f)^k(q(B))=\chi$. Thus, there exist $A\in\mathcal{U}$ such that $q(A)=\chi$ and by Remark \ref{remark1}, $q\circ F_n(f)^k(B)=\chi$. By Remark \ref{remark2}, $F_n(f)^k(B)=A$, which implies that $F_n(f)^k(B)\in \mathcal{U}$. Therefore, $B$ is a transitive point of $F_n(f)$.

\end{proof}

\begin{theorem}
Let $X$ be a compactum, let $n\geq2$, let $B\in F_n(X)$, and let $f: X \rightarrow X$ be a function. Consider the following statements:

(1)there exists $x\in X$ such that $\omega(x,f)=X$.

(2)there exists $A\in F_n(X)$ such that $\omega(A,F_n(f))=F_n(X)$.

(3)there exists $\varphi\in SF_n(X)$ such that $\omega(\varphi,SF_n(f))=SF_n(X)$.

Then the following hold: (2) and (3) are equivalent, (2) implies (1).
\end{theorem}
\begin{proof}
By \cite[Theorem 4.8]{a16}, we need only to prove that (3) implies (2).
 Suppose that there exists $\varphi\in SF_n(X)$ such that $\omega(\varphi,SF_n(f))=SF_n(X)$, it is easy to see that $\varphi\neq F_X$, so $q^{-1}(\varphi)\in F_n(X)$ and $q^{-1}(\varphi)\notin F_1(X)$. Next we show that $\omega(q^{-1}(\varphi),F_n(f))=F_n(X)$. Let $B\in F_n(X)$, let $k\in\mathbb{N}$ and let $\mathcal{U}$ be an open subset of $F_n(X)$ with $B\in \mathcal{U}$. By \cite[Lemma4.2]{a4}, there exist nonempty open subsets $U_1,U_2,...,U_n$ of $X$ such that $\langle U_1,U_2,...,U_n\rangle_n\subseteq \mathcal{U}$. For every $i\in\{1,2,...,n\}$, let $W_i$ be a nonempty open subset of $X$ such that $W_i\subseteq U_i$ and for any $i,j\in\{1,2,...,n\}$, $W_i\cap W_j\neq\emptyset$ if $i\neq j$. Note that $\langle U_1,U_2,...,U_n\rangle_n$ is nonempty open subsets of $F_n(X)$ such that $\langle W_1,W_2,...,W_n\rangle_n\subseteq \langle U_1,U_2,...,U_n\rangle_n\subseteq \mathcal{U}$ and $\langle W_1,W_2,...,W_n\rangle_n\cap F_1(X)=\emptyset$. By Remark \ref{remark2}, $q(\langle W_1,W_2,...,W_n\rangle_n)$ is nonempty open subsets of $SF_n(X)$ and $F_X\notin q(\langle W_1,W_2,...,W_n\rangle_n)$. Since $\omega(\varphi,SF_n(f))=SF_n(X)$, there exist $m\geq k$ such that $SF_n(f)^k(\varphi)\in q(\langle W_1,W_2,...,W_n\rangle_n)$, i.e., $SF_n(f)^m(q\cdot q^{-1}(\varphi))\in q(\langle W_1,W_2,...,W_n\rangle_n)$. By Remark \ref{remark1}, we have that $q\cdot F_n(f)^m( q^{-1}(\varphi))\in q(\langle W_1,W_2,...,W_n\rangle_n)$. By Remark \ref{remark2}, we obtain that $F_n(f)^m( q^{-1}(\varphi))\in \langle W_1,W_2,...,W_n\rangle_n$. Therefore, $F_n(f)^m(q^{-1}(\varphi))\in \mathcal{U}$, which implies that $\omega(q^{-1}(\varphi),F_n(f))=F_n(X)$.

\end{proof}

\begin{theorem}
Let $X$ be a compactum, let $n\geq2$, let $B\in F_n(X)$, and let $f: X \rightarrow X$ be a function. Consider the following statements:

(1)$trans(f)$ is dense in $X$.

(2)$trans(F_n(f))$ is dense in $F_n(X)$.

(3)$trans(SF_n(f))$ is dense in $SF_n(X)$.

Then (2) and (3) are equivalent, (2) implies (1), but (1) does not imply (2).
\end{theorem}
\begin{proof}
By \cite[Theorem 4.11]{a16}, we need only to prove that (3) implies (2) and that (1) does not imply (2).

Let $\mathcal{U}$ be an open subset of $F_n(X)$. By \cite[Lemma4.2]{a4}, there exist nonempty open subsets $U_1,U_2,...,U_n$ of $X$ such that $\langle U_1,U_2,...,U_n\rangle_n\subseteq \mathcal{U}$. For every $i\in\{1,2,...,n\}$, let $W_i$ be a nonempty open subset of $X$ such that $W_i\subseteq U_i$ and for any $i,j\in\{1,2,...,n\}$, $W_i\cap W_j\neq\emptyset$ if $i\neq j$. Note that $\langle U_1,U_2,...,U_n\rangle_n$ is nonempty open subsets of $F_n(X)$ such that $\langle W_1,W_2,...,W_n\rangle_n\subseteq \langle U_1,U_2,...,U_n\rangle_n\subseteq \mathcal{U}$ and $\langle W_1,W_2,...,W_n\rangle_n\cap F_1(X)=\emptyset$. By Remark \ref{remark2}, $q(\langle W_1,W_2,...,W_n\rangle_n)$ is nonempty open subsets of $SF_n(X)$ and $F_X\notin q(\langle W_1,W_2,...,W_n\rangle_n)$. Since $trans(SF_n(f))$ is dense in $SF_n(X)$, $trans(SF_n(f))\cap q(\langle W_1,W_2,...,W_n\rangle_n)\neq\emptyset$. Let $\chi\in trans(SF_n(f))\cap q(\langle W_1,W_2,...,W_n\rangle_n)$. Hence, there exists $A\in \langle W_1,W_2,...,W_n\rangle_n$ such that $\chi=q(A)$. Since $\langle W_1,W_2,...,W_n\rangle_n\subseteq \mathcal{U}$, $A\in \mathcal{U}$. On the other hand, by Theorem \ref{dltr}, we have $A\in trans(F_n(f))$. Thus, $trans(SF_n(f))\cap \mathcal{U}\neq\emptyset$. Since $\mathcal{U}$ is arbitrary we have that $trans(F_n(f))$ is dense in $F_n(X)$.

By \cite[Remark 4.3]{a16}, we have that $trans(f)$ is dense in $X$ but $trans(SF_n(f))=\emptyset$, which shows that (1) does not imply (2).

\end{proof}

\begin{theorem}
Let $Y$ be a topological space, let $X$ be a compactum, let $n\geq2$, let $f: X \rightarrow X$  and $g: Y \rightarrow Y$ be functions. Consider the following statements:

(1)$f\times g$ is transitive.

(2)$F_n(f)\times g$ is transitive.

(3)$SF_n(f)\times g$ is transitive.

Then (2) and (3) are equivalent, (2) implies (1), but (1) does not imply (2).
\end{theorem}
\begin{proof}
By \cite[Theorem 4.12]{a16}, we need only to prove that (3) implies (2) and that (1) does not imply (2).
Let $\mathcal{U}\times U$ and $\mathcal{V}\times V$ be nonempty open subsets of $F_n(X)\times Y$.
By \cite[Lemma4.2]{a4}, there exist nonempty open subsets $U_1,U_2,...,U_n$ and $V_1,V_2,...,V_n$ of $X$ such that $\langle U_1,U_2,...,U_n\rangle_n\subseteq \mathcal{U}$ and $\langle V_1,V_2,...,V_n\rangle_n\subseteq \mathcal{V}$. For every $i\in\{1,2,...,n\}$, let $W_i$ be a nonempty open subset of $X$ such that $W_i\subseteq U_i$ and for any $i,j\in\{1,2,...,n\}$, $W_i\cap W_j\neq\emptyset$ if $i\neq j$. Similarly, for each $i\in\{1,2,...,n\}$, let $O_i$ be a nonempty open subset of $X$ such that $O_i\subseteq V_i$ and for any $i,j\in\{1,2,...,n\}$, $O_i\cap O_j\neq\emptyset$ if $i\neq j$. Note that $\langle U_1,U_2,...,U_n\rangle_n\times U$ and $\langle V_1,V_2,...,V_n\rangle_n\times V$ are nonempty open subsets of $F_n(X)\times Y$ such that $\langle W_1,W_2,...,W_n\rangle_n\subseteq \langle U_1,U_2,...,U_n\rangle_n\subseteq \mathcal{U}$, $\langle O_1,O_2,...,O_n\rangle_n\subseteq \langle V_1,V_2,...,V_n\rangle_n\subseteq \mathcal{V}$, $\langle W_1,W_2,...,W_n\rangle_n\cap F_1(X)=\emptyset$, and $\langle O_1,O_2,...,O_n\rangle_n\cap F_1(X)=\emptyset$. By Remark \ref{remark2}, $q(\langle W_1,W_2,...,W_n\rangle_n)\times U$ and $q(\langle O_1,O_2,...,O_n\rangle_n)\times V$ are nonempty open subsets of $SF_n(X)\times Y$ such that $F_X\notin q(\langle W_1,W_2,...,W_n\rangle_n)$ and $F_X\notin q(\langle O_1,O_2,...,O_n\rangle_n)$.
Since $SF_n(f)\times g$ is transitive, there exist $k\in\mathbb{N}$ and $(\chi,y)\in q(\langle W_1,W_2,...,W_n\rangle_n)\times U$ such that $(SF_n(f)\times g)^k((\chi,y))\in q(\langle O_1,O_2,...,O_n\rangle_n)\times V$. Further, there exist $A\in\mathcal{U}$ and $B\in\mathcal{V}$ such that $q(A)=\chi$ and $SF_n(f)^k(\chi)=q(B)$. By Remark \ref{remark1}, we have $q\circ F_n(f)^k(A)=q(B)$. By Remark \ref{remark2}, $F_n(f)^k(A)=B$, which implies that $F_n(f)^k(A)\in \mathcal{V}$. Therefore, $(F_n(f)\times g)^k((A,y))\in \mathcal{V}\times V$. Thus, $F_n(f)\times g$ is transitive.

Let $f$ be the irrational rotation map and let $g$ be mixing map, then $f\times g$ is transitive. But $F_n(f)$ is not transitive, further, $F_n(f)\times g$ is not transitive. Therefore, (1) does not imply (2).

\end{proof}

\begin{theorem}
Let $X$ be a compactum, let $n\geq2$, let $A\in F_n(X)$, and let $f: X \rightarrow X$ be a function. Consider the following statements:

(1)$f$ is an $F$-system.

(2)$F_n(f)$ is an $F$-system.

(3)$SF_n(f)$ is an $F$-system.

Then the following hold: (2) and (3) are equivalent, (2) implies (1).
\end{theorem}
\begin{proof}
By \cite[Theorem 4.14]{a16}, we need only to prove that (3) implies (2). By \cite[Theorem 4.12]{a5} and \cite[Lemma 4.16]{a5}, we get the result.

\end{proof}

\begin{theorem}
Let $X$ be a compactum, let $n\geq2$, let $B\in F_n(X)$, and let $f: X \rightarrow X$ be a function. Consider the following statements:

(1)$f$ is $TT_{++}$.

(2)$F_n(f)$ is $TT_{++}$.

(3)$SF_n(f)$ is $TT_{++}$.

Then (2) and (3) are equivalent, (2) implies (1), but (1) does not imply (2).
\end{theorem}
\begin{proof}
By \cite[Theorem 4.17]{a16}, we need only to prove that (3) implies (2).
Let $\mathcal{U},\mathcal{V}$ be nonempty open subsets of $F_n(X)$. By \cite[Lemma4.2]{a4}, there exist nonempty open subsets $U_1,U_2,...,U_n$ and $V_1,V_2,...,V_n$ of $X$ such that $\langle U_1,U_2,...,U_n\rangle_n\subseteq \mathcal{U}$ and $\langle V_1,V_2,...,V_n\rangle_n\subseteq \mathcal{V}$. For every $i\in\{1,2,...,n\}$, let $W_i$ be a nonempty open subset of $X$ such that $W_i\subseteq U_i$ and for any $i,j\in\{1,2,...,n\}$, $W_i\cap W_j\neq\emptyset$ if $i\neq j$. Similarly, for each $i\in\{1,2,...,n\}$, let $O_i$ be a nonempty open subset of $X$ such that $O_i\subseteq V_i$ and for any $i,j\in\{1,2,...,n\}$, $O_i\cap O_j\neq\emptyset$ if $i\neq j$. Note that $\langle U_1,U_2,...,U_n\rangle_n$ and $\langle V_1,V_2,...,V_n\rangle_n$ are nonempty open subsets of $F_n(X)$ such that $\langle W_1,W_2,...,W_n\rangle_n\subseteq \langle U_1,U_2,...,U_n\rangle_n\subseteq \mathcal{U}$, $\langle O_1,O_2,...,O_n\rangle_n\subseteq \langle V_1,V_2,...,V_n\rangle_n\subseteq \mathcal{V}$, $\langle W_1,W_2,...,W_n\rangle_n\cap F_1(X)=\emptyset$, and $\langle O_1,O_2,...,O_n\rangle_n\cap F_1(X)=\emptyset$. By Remark \ref{remark2}, $q(\langle W_1,W_2,...,W_n\rangle_n)$ and $q(\langle O_1,O_2,...,O_n\rangle_n)$ are nonempty open subsets of $SF_n(X)$ such that $F_X\notin q(\langle W_1,W_2,...,W_n\rangle_n)$ and $F_X\notin q(\langle O_1,O_2,...,O_n\rangle_n)$. Since $SF_n(f)$ is $TT_{++}$, there exists an infinite set $n_{SF_n(f)}(q(\langle W_1,W_2,...,W_n\rangle_n),q(\langle O_1,O_2,...,O_n\rangle_n))$. Let $k\in n_{SF_n(f)}(q(\langle W_1,W_2,...,W_n\rangle_n),q(\langle O_1,O_2,...,O_n\rangle_n))$, then $$SF_n(f)^k(q(\langle W_1,W_2,...,W_n\rangle_n))\cap q(\langle O_1,O_2,...,O_n\rangle_n)\neq\emptyset.$$ By Remark \ref{remark1}, it follows that $q\circ F_n(f)^k(\langle W_1,W_2,...,W_n\rangle_n)\cap q(\langle O_1,O_2,...,O_n\rangle_n)\neq\emptyset$. From Remark \ref{remark2}, we obtain that $F_n(f)^k(\langle W_1,W_2,...,W_n\rangle_n)\cap \langle O_1,O_2,...,O_n\rangle_n\neq\emptyset$, which implies that $F_n(f)^k(\mathcal{U})\cap \mathcal{V}\neq\emptyset$. Therefore, $k\in n_{F_n(f)}(\mathcal{U},\mathcal{V})$. Further, $n_{F_n(f)}(\mathcal{U},\mathcal{V})$ is infinite. Thus, $F_n(f)$ is $TT_{++}$.

\end{proof}

\begin{theorem}
Let $X$ be a compactum, let $n\geq2$, let $B\in F_n(X)$, and let $f: X \rightarrow X$ be a function. Consider the following statements:

(1)$f$ is Touhey.

(2)$F_n(f)$ is Touhey.

(3)$SF_n(f)$ is Touhey.

Then (2) and (3) are equivalent, (2) implies (1), but (1) does not imply (2).
\end{theorem}
\begin{proof}
By \cite[Theorem 4.18]{a16}, we need only to prove that (3) implies (2).
Let $\mathcal{U},\mathcal{V}$ be nonempty open subsets of $F_n(X)$. By \cite[Lemma4.2]{a4}, there exist nonempty open subsets $U_1,U_2,...,U_n$ and $V_1,V_2,...,V_n$ of $X$ such that $\langle U_1,U_2,...,U_n\rangle_n\subseteq \mathcal{U}$ and $\langle V_1,V_2,...,V_n\rangle_n\subseteq \mathcal{V}$. For every $i\in\{1,2,...,n\}$, let $W_i$ be a nonempty open subset of $X$ such that $W_i\subseteq U_i$ and for any $i,j\in\{1,2,...,n\}$, $W_i\cap W_j\neq\emptyset$ if $i\neq j$. Similarly, for each $i\in\{1,2,...,n\}$, let $O_i$ be a nonempty open subset of $X$ such that $O_i\subseteq V_i$ and for any $i,j\in\{1,2,...,n\}$, $O_i\cap O_j\neq\emptyset$ if $i\neq j$. Note that $\langle U_1,U_2,...,U_n\rangle_n$ and $\langle V_1,V_2,...,V_n\rangle_n$ are nonempty open subsets of $F_n(X)$ such that $\langle W_1,W_2,...,W_n\rangle_n\subseteq \langle U_1,U_2,...,U_n\rangle_n\subseteq \mathcal{U}$, $\langle O_1,O_2,...,O_n\rangle_n\subseteq \langle V_1,V_2,...,V_n\rangle_n\subseteq \mathcal{V}$, $\langle W_1,W_2,...,W_n\rangle_n\cap F_1(X)=\emptyset$, and $\langle O_1,O_2,...,O_n\rangle_n\cap F_1(X)=\emptyset$. By Remark \ref{remark2}, $q(\langle W_1,W_2,...,W_n\rangle_n)$ and $q(\langle O_1,O_2,...,O_n\rangle_n)$ are nonempty open subsets of $SF_n(X)$ such that $F_X\notin q(\langle W_1,W_2,...,W_n\rangle_n)$ and $F_X\notin q(\langle O_1,O_2,...,O_n\rangle_n)$. Since $SF_n(f)$ is Touhey, there exist a periodic point $\chi\in  q(\langle W_1,W_2,...,W_n\rangle_n)$ and $k\in\mathbb{Z_+}$ such that $SF_n(f)^k(\chi)\in q(\langle O_1,O_2,...,O_n\rangle_n)$.
By Remark \ref{remark1}, we have that $q\circ F_n(f)^k(q^{-1}(\chi))\in q(\langle O_1,O_2,...,O_n\rangle_n)$. By Remark \ref{remark2}, $F_n(f)^k(q^{-1}(\chi))\in \langle O_1,O_2,...,O_n\rangle_n$. Further, we have that $q^{-1}(\chi)\in \mathcal{U}$ and $F_n(f)^k(q^{-1}(\chi))\in \mathcal{V}$. On the other hand, by \cite[Theorem 4.10]{a16}, $q^{-1}(\chi)$ is a periodic point of $F_n(f)$. Therefore $F_n(f)$ is Touhey.

\end{proof}

\begin{theorem}
Let $X$ be a compactum, let $n\geq2$, let $B\in F_n(X)$, and let $f: X \rightarrow X$ be a function. Consider the following statements:

(1)$f$ is two-sided transitive.

(2)$F_n(f)$ is two-sided transitive.

(3)$SF_n(f)$ is two-sided transitive.

Then (2) and (3) are equivalent, (2) implies (1).
\end{theorem}
\begin{proof}
By \cite[Theorem 4.21]{a16}, we need only to prove that (3) implies (2).
Suppose that $SF_n(f)$ is two-sided transitive. Then $SF_n(f)$ is a homeomorphism and there exists $\chi\in SF_n(X)$ such that $\{SF_n(f)^l(\chi):l\in \mathbb{Z}\}$ is dense in $SF_n(X)$. Note that $\chi\neq F_X$, so there exists $A\in F_n(X)\setminus F_1(X)$ such that $p(A)=\chi$. By \cite[Theorem 4.2]{a16}, we have that $F_n(f)$ is a homeomorphism. Let $\mathcal{U}$ be nonempty open subset of $F_n(X)$. By \cite[Lemma4.2]{a4}, there exist nonempty open subsets $U_1,U_2,...,U_n$ of $X$ such that $\langle U_1,U_2,...,U_n\rangle_n\subseteq \mathcal{U}$. For every $i\in\{1,2,...,n\}$, let $W_i$ be a nonempty open subset of $X$ such that $W_i\subseteq U_i$ and for any $i,j\in\{1,2,...,n\}$, $W_i\cap W_j\neq\emptyset$ if $i\neq j$. Note that $\langle U_1,U_2,...,U_n\rangle_n$ is nonempty open subsets of $F_n(X)$ such that $\langle W_1,W_2,...,W_n\rangle_n\subseteq \langle U_1,U_2,...,U_n\rangle_n\subseteq \mathcal{U}$ and $\langle W_1,W_2,...,W_n\rangle_n\cap F_1(X)=\emptyset$. By Remark \ref{remark2}, $q(\langle W_1,W_2,...,W_n\rangle_n)$ is nonempty open subsets of $SF_n(X)$ and $F_X\notin q(\langle W_1,W_2,...,W_n\rangle_n)$. By hypothesis, there exists $k\in\mathbb{Z}$ such that $SF_n(f)^k(\chi)\in q(\langle W_1,W_2,...,W_n\rangle_n)$.
By Remark \ref{remark1}, we have that $q\circ F_n(f)^k(q^{-1}(\chi))\in q(\langle W_1,W_2,...,W_n\rangle_n)$. By Remark \ref{remark2}, $F_n(f)^k(q^{-1}(\chi))\in \langle W_1,W_2,...,W_n\rangle_n$. Further, we have that $F_n(f)^k(q^{-1}(\chi))\in \mathcal{U}$. Consequently, $\{SF_n(f)^l(\chi):l\in \mathbb{Z}\}\cap \mathcal{U}\neq\emptyset$. Therefore, $F_n(f)$ is two-sided transitive.

\end{proof}

The following two Theorems extend Theorem 4.14 and Theorem 5.12 in \cite{a6}.
\begin{theorem}
Let $X$ be a compactum, let $n\geq2$, let $A\in F_n(X)$, and let $f: X \rightarrow X$ be a function. Consider the following statements:

(1)$f$ is fully exact.

(2)$F_n(f)$ is fully exact.

(3)$SF_n(f)$ is fully exact.

Then (2) and (3) are equivalent, and (2) implies (1).
\end{theorem}
\begin{proof}
By \cite[Theorem 4.3]{a6}, we need only to prove that (3) implies (2).
Let $\mathcal{U},\mathcal{V}$ be nonempty open subsets of $F_n(X)$. By \cite[Lemma4.2]{a4}, there exist nonempty open subsets $U_1,U_2,...,U_n$ and $V_1,V_2,...,V_n$ of $X$ such that $\langle U_1,U_2,...,U_n\rangle_n\subseteq \mathcal{U}$ and $\langle V_1,V_2,...,V_n\rangle_n\subseteq \mathcal{V}$. For every $i\in\{1,2,...,n\}$, let $W_i$ be a nonempty open subset of $X$ such that $W_i\subseteq U_i$ and for any $i,j\in\{1,2,...,n\}$, $W_i\cap W_j\neq\emptyset$ if $i\neq j$. Similarly, for each $i\in\{1,2,...,n\}$, let $O_i$ be a nonempty open subset of $X$ such that $O_i\subseteq V_i$ and for any $i,j\in\{1,2,...,n\}$, $O_i\cap O_j\neq\emptyset$ if $i\neq j$. Note that $\langle U_1,U_2,...,U_n\rangle_n$ and $\langle V_1,V_2,...,V_n\rangle_n$ are nonempty open subsets of $F_n(X)$ such that $\langle W_1,W_2,...,W_n\rangle_n\subseteq \langle U_1,U_2,...,U_n\rangle_n\subseteq \mathcal{U}$, $\langle O_1,O_2,...,O_n\rangle_n\subseteq \langle V_1,V_2,...,V_n\rangle_n\subseteq \mathcal{V}$, $\langle W_1,W_2,...,W_n\rangle_n\cap F_1(X)=\emptyset$, and $\langle O_1,O_2,...,O_n\rangle_n\cap F_1(X)=\emptyset$. By Remark \ref{remark2}, $q(\langle W_1,W_2,...,W_n\rangle_n)$ and $q(\langle O_1,O_2,...,O_n\rangle_n)$ are nonempty open subsets of $SF_n(X)$ such that $F_X\notin q(\langle W_1,W_2,...,W_n\rangle_n)$ and $F_X\notin q(\langle O_1,O_2,...,O_n\rangle_n)$.
 Since $SF_n(f)$ is fully exact, there exists $k\in\mathbb{N}$ such that $int[SF_n(f)^k(q(\langle W_1,W_2,...,W_n\rangle_n))\cap SF_n(f)^k(q(\langle W_1,W_2,...,W_n\rangle_n))]\neq \emptyset$. Thus, there exists a nonempty open subset $\Omega$ of $SF_n(X)$ such that $\Omega\subseteq SF_n(f)^k(q(\langle W_1,W_2,...,W_n\rangle_n))\cap SF_n(f)^k(q(\langle W_1,W_2,...,W_n\rangle_n))$.
By Remark \ref{remark1}, we have that $\Omega\subseteq q\circ F_n(f)^k(\langle W_1,W_2,...,W_n\rangle_n)\cap q\circ F_n(f)^k(\langle W_1,W_2,...,W_n\rangle_n)$ and that $F_X\notin q\circ F_n(f)^k(\langle W_1,W_2,...,W_n\rangle_n)\cap q\circ F_n(f)^k(\langle W_1,W_2,...,W_n\rangle_n)$.
Therefore,
\begin{equation}
\begin{aligned}
q^{-1}(\Omega)\subseteq& q^{-1}\circ (q\circ F_n(f)^k(\langle W_1,W_2,...,W_n\rangle_n)\cap q\circ F_n(f)^k(\langle W_1,W_2,...,W_n\rangle_n))\\
\subseteq&F_n(f)^k(\langle W_1,W_2,...,W_n\rangle_n)\cap F_n(f)^k(\langle W_1,W_2,...,W_n\rangle_n)\\
\subseteq&F_n(f)^k(\mathcal{U})\cap F_n(f)^k(\mathcal{V}).
\end{aligned}
\end{equation}
Since $q^{-1}$ is continuous, $q^{-1}(\Omega)$ is a nonempty open subsets of $F_n(X)$. Thus, $int(F_n(f)^k(\mathcal{U})\cap F_n(f)^k(\mathcal{V}))\neq \emptyset)$.  Therefore $F_n(f)$ is fully exact.

\end{proof}

\begin{theorem}
Let $X$ be a compactum, let $n\geq2$, let $A\in F_n(X)$, and let $f: X \rightarrow X$ be a function. Consider the following statements:

(1)$f$ is strongly transitive.

(2)$F_n(f)$ is strongly transitive.

(3)$SF_n(f)$ is strongly transitive.

Then (2) implies (3), (3) implies (1), but (1) does not imply (2) or (3).
\end{theorem}
\begin{proof}
By \cite[Theorem 5.12]{a6}, we need only to prove that (1) does not imply (3).

Let $f$ be the irrational rotation map, then $f$ is strongly transitive, but $SF_n(f)$ is not weakly mixing. Thus, $SF_n(f)$ is not transitive by Theorem \ref{thdj}, which implies that $SF_n(f)$ is not strongly transitive.

\end{proof}

%\section*{Acknowledgments}

%\bibliography{mybibfile}

\end{document}